\documentclass[twoside,11pt,reqno]{amsart}
\usepackage{amsmath,amsthm,amssymb,amstext,amsfonts,amscd}
\usepackage{graphicx}
\usepackage{multirow}

\newtheorem{theorem}{Theorem}

\newtheorem{definition}{Definition}

\newtheorem{lemma}{Lemma}

\newtheorem{remark}{Remark}

\numberwithin{equation}{section}
\input{tcilatex}

\begin{document}
\title[ On the integral representations .....]{On the integral
representations of relative $\left( p,q\right) $\ -th type and relative $%
\left( p,q\right) $\ -th weak type of entire\ and meromorphic functions}
\author[Tanmay Biswas]{Tanmay Biswas}
\address{T. Biswas : Rajbari, Rabindrapalli, R. N. Tagore Road, P.O.-
Krishnagar, Dist-Nadia, PIN-\ 741101, West Bengal, India}
\email{tanmaybiswas\_math@rediffmail.com}

\begin{abstract}
{\small In this paper we wish to establish the integral representations of
relative }$\left( p,q\right) ${\small \ -th type and relative }$\left(
p,q\right) ${\small \ -th weak type of entire and meromorphic functions. We
also investigate their equivalence relation under some certain condition.}
\end{abstract}

\maketitle

\section{\textbf{Introduction}}

\qquad For any entire function $f$\ , $M_{f}\left( r\right) $, a function of
$r$\ is defined as follows:%
\begin{equation*}
M\left( r\right) \equiv M_{f}\left( r\right) =\text{ }\underset{\left\vert
z\right\vert =r}{\max }|f\left( z\right) |.
\end{equation*}

\qquad If an entire function $f$ is non-constant then $M_{f}\left( r\right) $
is strictly increasing and continuous and its inverse $M_{f}{}^{-1}:\left(
\left\vert f\left( 0\right) \right\vert ,\infty \right) \rightarrow \left(
0,\infty \right) $ exists and is such that $\underset{s\rightarrow \infty }{%
\lim }M_{f}^{-1}\left( s\right) =\infty .$

\qquad Whenever $f$ is meromorphic, one can define another function $%
T_{f}\left( r\right) $ known as Nevanlinna's Characteristic function of $f$
in the following manner which perform the same role as the maximum modulus
function:%
\begin{equation*}
T_{f}\left( r\right) =N_{f}\left( r\right) +m_{f}\left( r\right) ,
\end{equation*}

\qquad wherever the function $N_{f}\left( r,a\right) \left( \overset{-}{N_{f}%
}\left( r,a\right) \right) $ known as counting function of\ $a$-points
(distinct $a$-points) of meromorphic $f$ is defined as%
\begin{equation*}
N_{f}\left( r,a\right) =\overset{r}{\underset{0}{\int }}\frac{n_{f}\left(
t,a\right) -n_{f}\left( 0,a\right) }{t}dt+\overset{-}{n_{f}}\left(
0,a\right) \log r
\end{equation*}%
\begin{equation*}
\left( \overset{-}{N_{f}}\left( r,a\right) =\overset{r}{\underset{0}{\int }}%
\frac{\overset{-}{n_{f}}\left( t,a\right) -\overset{-}{n_{f}}\left(
0,a\right) }{t}dt+\overset{-}{n_{f}}\left( 0,a\right) \log r~\right) ,
\end{equation*}%
in addition we symbolize by $n_{f}\left( r,a\right) \left( \overset{-}{n_{f}}%
\left( r,a\right) \right) $ the number of $a$-points (distinct $a$-points)
of $f$ in $\left\vert z\right\vert \leq r$ and an $\infty $ -point is a pole
of $f$. In many situations, $N_{f}\left( r,\infty \right) $ and $\overset{-}{%
N_{f}}\left( r,\infty \right) $ are symbolized by $N_{f}\left( r\right) $
and $\overset{-}{N_{f}}\left( r\right) $ respectively. Also the function $%
m_{f}\left( r,\infty \right) $ alternatively symbolized by $m_{f}\left(
r\right) $ known as the proximity function of $f$ is defined in the
following way:%
\begin{align*}
m_{f}\left( r\right) & =\frac{1}{2\pi }\overset{2\pi }{\underset{0}{\int }}%
\log ^{+}\left\vert f\left( re^{i\theta }\right) \right\vert d\theta ,~\ ~%
\text{where} \\
\log ^{+}x& =\max \left( \log x,0\right) \text{ for all }x\geqslant 0,
\end{align*}%
and some times one may denote $m\left( r,\frac{1}{f-a}\right) $ by $%
m_{f}\left( r,a\right) $.

\qquad The term $m\left( r,a\right) $ which is defined to be the mean value
of $\log^{+}\left\vert \frac{1}{f-a}\right\vert $ ( or $\log^{+}\left\vert
f\right\vert $ if $a=\infty$) on the circle $\left\vert z\right\vert =r,$
receives a remarkable contribution only from those arcs on the circle where
the functional values differ very little from the given value `$a$'. The
magnitude of the proximity function can thus be considered as a measure for
the mean deviation on the circle $\left\vert z\right\vert =r$ of the
functional value $f$ from the value `$a$'.

\qquad If $f$ is an entire function, then the Nevanlinna's Characteristic
function $T_{f}\left( r\right) $ of $f$ is defined as follows:%
\begin{equation*}
T_{f}\left( r\right) =m_{f}\left( r\right) ~.
\end{equation*}

\qquad Moreover, if $f$ is non-constant entire then $T_{f}\left( r\right) $
is strictly increasing and continuous functions of $r$. Also its inverse $%
T_{f}^{-1}:\left( T_{f}\left( 0\right) ,\infty \right) \rightarrow \left(
0,\infty \right) $ exist and is such that $\underset{s\rightarrow \infty }{%
\lim }T_{f}^{-1}\left( s\right) =\infty $. Also the ratio $\frac{T_{f}\left(
r\right) }{T_{g}\left( r\right) }$ as $r\rightarrow \infty $ is called the
growth of $f$ with respect to $g$ in terms of the Nevanlinna's
Characteristic functions of the meromorphic functions $f$ and $g$.

\qquad The $\emph{order}$ and \emph{lower order} of an entire function $f$
which is generally used in computational purpose are classical in complex
analysis. L. Bernal ,$\left\{ \text{\cite{1},\cite{2}}\right\} $ introduced
the \emph{relative order} (respectively \emph{relative lower order)\ }%
between two entire functions to avoid comparing growth just with $\exp z.$
Extending the notion of \emph{relative order} (respectively \emph{relative
lower order) }Ruiz et al \cite{5} introduced the \emph{relative }$\left(
p,q\right) $\emph{-th order} (respectively \emph{relative\ lower }$\left(
p,q\right) $\emph{-th order) }where $p$\ and $q$ are any two positive
integers. Now to compare the growth of entire functions having the same
\emph{relative }$\left( p,q\right) $\emph{-th order} or \emph{relative\
lower }$\left( p,q\right) $\emph{-th order}, we wish to introduce the
definition of \emph{relative }$\left( p,q\right) $\emph{\ -th type }and\emph{%
\ relative }$\left( p,q\right) $\emph{\ -th weak type} of an entire function
with respect to another entire function and establish their integral
representations. We also investigate their equivalence relations under
certain conditions. We do not explain the standard definitions and notations
in the theory of entire and meromorphic functions as those are available in
\cite{4} and \cite{6}.

\section{\textbf{Preliminary remarks and definitions}}

\qquad First of all we state the following two notations which are
frequently used in our subsequent study:%
\begin{equation*}
\log ^{[k]}r=\log \left( \log ^{[k-1]}r\right) \text{ for }k=1,2,3,\cdot
\cdot \cdot \text{ };\text{ }\log ^{[0]}r=r
\end{equation*}%
and%
\begin{equation*}
\exp ^{[k]}r=\exp \left( \exp ^{[k-1]}r\right) \text{ for }k=1,2,3,\cdot
\cdot \cdot \text{ ; }\exp ^{[0]}r=r.
\end{equation*}

\qquad Taking this into account, let us denote that%
\begin{equation*}
\rho _{\alpha }^{\left( p,q\right) }\left( \beta \right) =\text{ }\underset{%
r\rightarrow \infty }{\lim \sup }\frac{\log ^{\left[ p\right] }\alpha
^{-1}\beta \left( r\right) }{\log ^{\left[ q\right] }r}\text{ and }\lambda
_{\alpha }^{\left( p,q\right) }\left( \beta \right) =\underset{r\rightarrow
\infty }{\text{ }\lim \inf }\frac{\log ^{\left[ p\right] }\alpha ^{-1}\beta
\left( r\right) }{\log ^{\left[ q\right] }r}
\end{equation*}%
where $p$,$q$ are any two positive integers and $\alpha \left( x\right)
,\beta \left( x\right) $ be any two positive continuous increasing to $%
+\infty $ on $[x_{0},+\infty )$ functions.

\qquad If we consider $\alpha (x)=M_{g}(x)$ and $\beta
(x)=M_{f}(x)$ where $f $ and $g$ are any two entire functions with
index-pairs $\left( m,q\right) $ and $\left( m,p\right) $
respectively where $p,q,m$ are positive integers such that $m\geq
\max (p,q),$ then the above definition reduces to the definition
of relative $\left( p,q\right) $-th order and relative $\left(
p,q\right) $-th lower order of an entire function $f$ with respect
to
another entire function $g$ respectively as introduced by Ruiz et. al. \cite%
{5}. Similarly if we take $\alpha (x)=T_{g}(x)$ and $\beta
(x)=T_{f}(x)$ where $f$ be a meromorphic function and $g$ be any
entire function with index-pairs $\left( m,q\right) $ and $\left(
m,p\right) $ respectively where $p,q,m$ are positive integers such
that $m\geq \max (p,q),$ then the above definition reduces to the
definition of relative $\left( p,q\right) $-th order and relative
$\left( p,q\right) $-th lower order of a meromorphic function $f$
with respect to an entire function $g$ respectively as introduced
by Debnath et. al. \cite{3}. For detains about index pair one may
see \cite{3} and \cite{5}.

\qquad In order to refine the above growth scale, now we intend to introduce
the definition of an another growth indicator, called \emph{relative }$%
\left( p,q\right) $\emph{\ -th type} which is as follows:

\begin{definition}
\label{d1.4} Let $\alpha \left( x\right) $ and $\beta \left( x\right) $ be
any two positive continuous increasing to $+\infty $ on $[x_{0},+\infty )$
functions. The relative $\left( p,q\right) $ th type of $\beta \left(
x\right) $ with respect to $\alpha \left( x\right) $ having finite positive
relative $\left( p,q\right) $ th order $\rho _{\alpha }^{\left( p,q\right)
}\left( \beta \right) $ $\left( a<\rho _{\alpha }^{\left( p,q\right) }\left(
\beta \right) <\infty \right) $ where $p$\ and $q$ are any two positive
integers is defined as :%
\begin{equation*}
\sigma _{\alpha }^{\left( p,q\right) }\left( \beta \right) =\underset{%
r\rightarrow \infty }{\lim \sup }\frac{\log ^{\left[ p-1\right] }\alpha
^{-1}\beta \left( r\right) }{\left( \log ^{\left[ q-1\right] }r\right)
^{\rho _{\alpha }^{\left( p,q\right) }\left( \beta \right) }}~.
\end{equation*}
\end{definition}

\qquad The above definition can alternatively defined in the following
manner:

\begin{definition}
\label{d1.5} Let $\alpha \left( x\right) $ and $\beta \left( x\right) $ be
any two positive continuous increasing to $+\infty $ on $[x_{0},+\infty )$
functions having finite positive \emph{relative }$\left( p,q\right) $\emph{\
-th order} $\rho _{\alpha }^{\left( p,q\right) }\left( \beta \right) $ $%
\left( a<\rho _{\alpha }^{\left( p,q\right) }\left( \beta \right) <\infty
\right) $ where $p$\ and $q$ are any two positive integers. Then the \emph{%
relative }$\left( p,q\right) $\emph{\ -th type} $\sigma _{\alpha }^{\left(
p,q\right) }\left( \beta \right) $ of $\beta \left( x\right) $ with respect
to $\alpha \left( x\right) $ is define as: The integral $\int%
\limits_{r_{0}}^{\infty }\frac{\log ^{\left[ p-2\right] }\alpha ^{-1}\beta
\left( r\right) }{\left[ \exp \left( \left( \log ^{\left[ q-1\right]
}r\right) ^{\rho _{\alpha }^{\left( p,q\right) }\left( \beta \right)
}\right) \right] ^{k+1}}dr\left( r_{0}>0\right) $ converges for $k>\sigma
_{\alpha }^{\left( p,q\right) }\left( \beta \right) $ and diverges for $%
k<\sigma _{\alpha }^{\left( p,q\right) }\left( \beta \right) .$
\end{definition}

\qquad Analogously, to determine the relative growth of two increasing
functions having same non zero finite \emph{relative }$\left( p,q\right) $%
\emph{\ -th lower order}, one can introduced the definition of \emph{%
relative }$\left( p,q\right) $\emph{\ -th weak type} of finite positive
\emph{relative }$\left( p,q\right) $\emph{\ -th lower order }$\lambda
_{\alpha }^{\left( p,q\right) }\left( \beta \right) $ in the following way:

\begin{definition}
\label{d1.6} Let $\alpha \left( x\right) $ and $\beta \left( x\right) $ be
any two positive continuous increasing to $+\infty $ on $[x_{0},+\infty )$
functions having finite positive relative $\left( p,q\right) $ th lower
order $\lambda _{\alpha }^{\left( p,q\right) }\left( \beta \right) $
\linebreak $\left( a<\lambda _{\alpha }^{\left( p,q\right) }\left( \beta
\right) <\infty \right) $ where $p$\ and $q$ are any two positive integers$.$%
Then the \emph{relative }$\left( p,q\right) $\emph{\ th weak type} of $\beta
\left( x\right) $ with respect to $\alpha \left( x\right) $ is defined as :%
\begin{equation*}
\tau _{\alpha }^{\left( p,q\right) }\left( \beta \right) =\underset{%
r\rightarrow \infty }{\lim \inf }\frac{\log ^{\left[ p-1\right] }\alpha
^{-1}\beta \left( r\right) }{\left( \log ^{\left[ q-1\right] }r\right)
^{\lambda _{\alpha }^{\left( p,q\right) }\left( \beta \right) }}~.
\end{equation*}
\end{definition}

\qquad The above definition can also alternatively defined as:

\begin{definition}
\label{d1.7} Let $\alpha \left( x\right) $ and $\beta \left( x\right) $ be
any two positive continuous increasing to $+\infty $ on $[x_{0},+\infty )$
functions having finite positive \emph{relative }$\left( p,q\right) $\emph{\
-th lower order} $\lambda _{\alpha }^{\left( p,q\right) }\left( \beta
\right) $ $\left( a<\lambda _{\alpha }^{\left( p,q\right) }\left( \beta
\right) <\infty \right) $ where $p$\ and $q$ are any two positive integers.
Then the \emph{relative }$\left( p,q\right) $\emph{\ -th weak type} $\tau
_{\alpha }^{\left( p,q\right) }\left( \beta \right) $ of $\beta \left(
x\right) $ with respect to $\alpha \left( x\right) $ is defined as:%
\begin{equation*}
\text{The integral}\int\limits_{r_{0}}^{\infty }\frac{\log ^{\left[ p-2%
\right] }\alpha ^{-1}\beta \left( r\right) }{\left[ \exp \left( \left( \log
^{\left[ q-1\right] }r\right) ^{\lambda _{\alpha }^{\left( p,q\right)
}\left( \beta \right) }\right) \right] ^{k+1}}dr\left( r_{0}>0\right)
\end{equation*}%
converges for $k>\tau _{\alpha }^{\left( p,q\right) }\left( \beta \right) $
and diverges for $k<\tau _{\alpha }^{\left( p,q\right) }\left( \beta \right)
.$
\end{definition}

\qquad Now a question may arise about the equivalence of the definitions of
\emph{relative }$\left( p,q\right) $\emph{\ -th type }and\emph{\ relative }$%
\left( p,q\right) $\emph{\ -th weak type} with their integral
representations. In the next section we would like to establish such
equivalence of Definition \ref{d1.4} and Definition \ref{d1.5}, and
Definition \ref{d1.6} and Definition \ref{d1.7} and also investigate some
growth properties related to \emph{relative }$\left( p,q\right) $\emph{\ -th
type }and\emph{\ relative }$\left( p,q\right) $\emph{\ -th weak type }of $%
\beta \left( x\right) $ with respect to $\alpha \left( x\right) .$

\section{\textbf{Lemmas.}}

\qquad In this section we present a lemma which will be needed in the sequel.

\begin{lemma}
\label{l1.1} Let the integral $\int\limits_{r_{0}}^{\infty }\frac{\log ^{%
\left[ p-2\right] }\alpha ^{-1}\beta \left( r\right) }{\left[ \exp \left(
\left( \log ^{\left[ q-1\right] }r\right) ^{A}\right) \right] ^{k+1}}%
dr\left( r_{0}>0\right) $ converges where $0<A<\infty .$ Then%
\begin{equation*}
\underset{r\rightarrow \infty }{\lim }\frac{\log ^{\left[ p-2\right] }\alpha
^{-1}\beta \left( r\right) }{\left[ \exp \left( \left( \log ^{\left[ q-1%
\right] }r\right) ^{A}\right) \right] ^{k}}=0~.
\end{equation*}
\end{lemma}

\begin{proof}
Since the integral $\int\limits_{r_{0}}^{\infty }\frac{\log ^{\left[ p-2%
\right] }\alpha ^{-1}\beta \left( r\right) }{\left[ \exp \left( \left( \log
^{\left[ q-1\right] }r\right) ^{A}\right) \right] ^{k+1}}dr\left(
r_{0}>0\right) $ converges, then%
\begin{equation*}
\int\limits_{r_{0}}^{\infty }\frac{\log ^{\left[ p-2\right] }\alpha
^{-1}\beta \left( r\right) }{\left[ \exp \left( \left( \log ^{\left[ q-1%
\right] }r\right) ^{A}\right) \right] ^{k+1}}dr<\varepsilon ,\text{ if }%
r_{0}>R\left( \varepsilon \right) ~.
\end{equation*}%
Therefore,%
\begin{equation*}
\int\limits_{r_{0}}^{\exp \left( \left( \log ^{\left[ q-1\right]
}r_{0}\right) ^{A}\right) +r_{0}}\frac{\log ^{\left[ p-2\right] }\alpha
^{-1}\beta \left( r\right) }{\left[ \exp \left( \left( \log ^{\left[ q-1%
\right] }r\right) ^{A}\right) \right] ^{k+1}}dr<\varepsilon ~.
\end{equation*}%
Since $\log ^{\left[ p-2\right] }\alpha ^{-1}\beta \left( r\right) $
increases with $r$, so%
\begin{equation*}
\int\limits_{r_{0}}^{\exp \left( \left( \log ^{\left[ q-1\right]
}r_{0}\right) ^{A}\right) +r_{0}}\frac{\log ^{\left[ p-2\right] }\alpha
^{-1}\beta \left( r\right) }{\left[ \exp \left( \left( \log ^{\left[ q-1%
\right] }r\right) ^{A}\right) \right] ^{k+1}}dr\geq \ \ \ \ \ \ \ \ \ \ \ \
\ \ \ \
\end{equation*}%
\begin{equation*}
~\ \ \ \ \ \ \ \ \ \ \frac{\log ^{\left[ p-2\right] }\alpha ^{-1}\beta
\left( r_{0}\right) }{\left[ \exp \left( \left( \log ^{\left[ q-1\right]
}r_{0}\right) ^{A}\right) \right] ^{k+1}}\cdot \left[ \exp \left( \left(
\log ^{\left[ q-1\right] }r_{0}\right) ^{A}\right) \right] ~.
\end{equation*}%
i.e., for all sufficiently large values of $r$,%
\begin{equation*}
\int\limits_{r_{0}}^{\exp \left( \left( \log ^{\left[ q-1\right]
}r_{0}\right) ^{A}\right) +r_{0}}\frac{\log ^{\left[ p-2\right] }\alpha
^{-1}\beta \left( r\right) }{\left[ \exp \left( \left( \log ^{\left[ q-1%
\right] }r\right) ^{A}\right) \right] ^{k+1}}dr\geq \ \ \ \ \ \ \ \ \ \ \ \
\ \ \
\end{equation*}%
\begin{equation*}
~\ \ \ \ ~\ \ \ \ \ \ \ \ \ \ \ \ \ \ \ \ \ \ \ \ \ \ \ \ \ \ \ \ \ \ \ \ \
\ \ \ \ \ \ \ \ \ \ \ \ \ \ \ \ \ \frac{\log ^{\left[ p-2\right] }\alpha
^{-1}\beta \left( r_{0}\right) }{\left[ \exp \left( \left( \log ^{\left[ q-1%
\right] }r_{0}\right) ^{A}\right) \right] ^{k}}~,
\end{equation*}%
so that%
\begin{equation*}
\frac{\log ^{\left[ p-2\right] }\alpha ^{-1}\beta \left( r_{0}\right) }{%
\left[ \exp \left( \left( \log ^{\left[ q-1\right] }r_{0}\right) ^{A}\right) %
\right] ^{k}}<\varepsilon \text{ if }r_{0}>R\left( \varepsilon \right) .
\end{equation*}%
\begin{equation*}
\underset{r\rightarrow \infty }{i.e.,~\lim }\frac{\log ^{\left[ p-2\right]
}\alpha ^{-1}\beta \left( r\right) }{\left[ \exp \left( \left( \log ^{\left[
q-1\right] }r\right) ^{A}\right) \right] ^{k}}=0.
\end{equation*}%
This proves the lemma.
\end{proof}

\section{\textbf{Main Results.}}

\qquad In this section we state the main results of this chapter.

\begin{theorem}
\label{t3.1} Let $\alpha \left( x\right) $ and $\beta \left( x\right) $ be
any two positive continuous increasing to $+\infty $ on $[x_{0},+\infty )$
functions having finite positive \emph{relative }$\left( p,q\right) $\emph{\
-th order} $\rho _{\alpha }^{\left( p,q\right) }\left( \beta \right) $ $%
\left( 0<\rho _{\alpha }^{\left( p,q\right) }\left( \beta \right) <\infty
\right) $ and \emph{relative }$\left( p,q\right) $\emph{\ -th type} $\sigma
_{\alpha }^{\left( p,q\right) }\left( \beta \right) $ where $p$\ and $q$ are
any two positive integers. Then Definition \ref{d1.4} and Definition \ref%
{d1.5} are equivalent.
\end{theorem}

\begin{proof}
Let us consider $\alpha \left( x\right) $ and $\beta \left( x\right) $ be
any two positive continuous increasing to $+\infty $ on $[x_{0},+\infty )$
functions such that $\rho _{\alpha }^{\left( p,q\right) }\left( \beta
\right) $ $\left( 0<\rho _{\alpha }^{\left( p,q\right) }\left( \beta \right)
<\infty \right) $ exists for any two positive integers $p$\ and $q$.\newline
\textbf{Case I.\ }$\sigma _{\alpha }^{\left( p,q\right) }\left( \beta
\right) =\infty .$\newline
\textbf{Definition \ref{d1.4} }$\Rightarrow $\textbf{\ Definition \ref{d1.5}.%
}\newline
\qquad As $\sigma _{\alpha }^{\left( p,q\right) }\left( \beta \right)
=\infty $, from Definition \ref{d1.4} we have for arbitrary positive $G$ and
for a sequence of values of $r$ tending to infinity that%
\begin{align}
\log ^{\left[ p-1\right] }\alpha ^{-1}\beta \left( r\right) & >G\cdot \left(
\log ^{\left[ q-1\right] }r\right) ^{\rho _{\alpha }^{\left( p,q\right)
}\left( \beta \right) }  \notag \\
i.e.,~\log ^{\left[ p-2\right] }\alpha ^{-1}\beta \left( r\right) & >\left[
\exp \left( \left( \log ^{\left[ q-1\right] }r\right) ^{\rho _{\alpha
}^{\left( p,q\right) }\left( \beta \right) }\right) \right] ^{G}~.
\label{3.1}
\end{align}%
If possible let the integral $\int\limits_{r_{0}}^{\infty }\frac{\log ^{%
\left[ p-2\right] }\alpha ^{-1}\beta \left( r\right) }{\left[ \exp \left(
\left( \log ^{\left[ q-1\right] }r\right) ^{\rho _{\alpha }^{\left(
p,q\right) }\left( \beta \right) }\right) \right] ^{G+1}}dr$ $\left(
r_{0}>0\right) $ be converge.\newline
Then by Lemma \ref{l1.1},%
\begin{equation*}
\underset{r\rightarrow \infty }{\lim \sup }\frac{\log ^{\left[ p-2\right]
}\alpha ^{-1}\beta \left( r\right) }{\left[ \exp \left( \left( \log ^{\left[
q-1\right] }r\right) ^{\rho _{\alpha }^{\left( p,q\right) }\left( \beta
\right) }\right) \right] ^{G}}=0~.
\end{equation*}%
So for all sufficiently large values of $r$,%
\begin{equation}
\log ^{\left[ p-2\right] }\alpha ^{-1}\beta \left( r\right) <\left[ \exp
\left( \left( \log ^{\left[ q-1\right] }r\right) ^{\rho _{\alpha }^{\left(
p,q\right) }\left( \beta \right) }\right) \right] ^{G}~.  \label{3.2}
\end{equation}%
Therefore from $\left( \ref{3.1}\right) $ and $\left( \ref{3.2}\right) $ we
arrive at a contradiction.

Hence $\int\limits_{r_{0}}^{\infty }\frac{\log ^{\left[ p-2\right] }\alpha
^{-1}\beta \left( r\right) }{\left[ \exp \left( \left( \log ^{\left[ q-1%
\right] }r\right) ^{\rho _{\alpha }^{\left( p,q\right) }\left( \beta \right)
}\right) \right] ^{G+1}}dr$ $\left( r_{0}>0\right) $ diverges whenever $G$
is finite, which is the Definition \ref{d1.5}.\newline
\textbf{Definition \ref{d1.5} }$\Rightarrow $\textbf{\ Definition \ref{d1.4}.%
}\newline
\qquad Let $G$ be any positive number. Since $\sigma _{\alpha }^{\left(
p,q\right) }\left( \beta \right) =\infty $, from Definition \ref{d1.5}, the
divergence of the integral $\int\limits_{r_{0}}^{\infty }\frac{\log ^{\left[
p-2\right] }\alpha ^{-1}\beta \left( r\right) }{\left[ \exp \left( \left(
\log ^{\left[ q-1\right] }r\right) ^{\rho _{\alpha }^{\left( p,q\right)
}\left( \beta \right) }\right) \right] ^{G+1}}dr$ $\left( r_{0}>0\right) $
gives for arbitrary positive $\varepsilon $ and for a sequence of values of $%
r$ tending to infinity%
\begin{align*}
\log ^{\left[ p-2\right] }\alpha ^{-1}\beta \left( r\right) & >\left[ \exp
\left( \left( \log ^{\left[ q-1\right] }r\right) ^{\rho _{\alpha }^{\left(
p,q\right) }\left( \beta \right) }\right) \right] ^{G-\varepsilon } \\
i.e.,~\log ^{\left[ p-1\right] }\alpha ^{-1}\beta \left( r\right) & >\left(
G-\varepsilon \right) \left( \log ^{\left[ q-1\right] }r\right) ^{\rho
_{\alpha }^{\left( p,q\right) }\left( \beta \right) },
\end{align*}%
which implies that%
\begin{equation*}
\underset{r\rightarrow \infty }{\lim \sup }\frac{\log ^{\left[ p-1\right]
}\alpha ^{-1}\beta \left( r\right) }{\left( \log ^{\left[ q-1\right]
}r\right) ^{\rho _{\alpha }^{\left( p,q\right) }\left( \beta \right) }}\geq
G-\varepsilon ~.
\end{equation*}%
Since $G>0$ is arbitrary, it follows that%
\begin{equation*}
\underset{r\rightarrow \infty }{\lim \sup }\frac{\log ^{\left[ p-1\right]
}\alpha ^{-1}\beta \left( r\right) }{\left( \log ^{\left[ q-1\right]
}r\right) ^{\rho _{\alpha }^{\left( p,q\right) }\left( \beta \right) }}%
=\infty ~.
\end{equation*}%
Thus Definition \ref{d1.4} follows.\newline

\textbf{Case II.\ }$0\leq \sigma _{\alpha }^{\left( p,q\right) }\left( \beta
\right) <\infty .$\newline
\textbf{Definition \ref{d1.4} }$\Rightarrow $\textbf{\ Definition \ref{d1.5}.%
}\newline
\textbf{Subcase (A).$\mathbb{\ }$\ }$0<\sigma _{\alpha }^{\left( p,q\right)
}\left( \beta \right) <\infty .$\newline

Let $\alpha \left( x\right) $ and $\beta \left( x\right) $ be any two
positive continuous increasing to $+\infty $ on $[x_{0},+\infty )$ functions
such that $0<\sigma _{\alpha }^{\left( p,q\right) }\left( \beta \right)
<\infty $ exists for any two positive integers $p$\ and $q.$ Then according
to the Definition \ref{d1.4}, for arbitrary positive $\varepsilon $ and for
all sufficiently large values of $r$, we obtain that%
\begin{align}
\log ^{\left[ p-1\right] }\alpha ^{-1}\beta \left( r\right) & <\left( \sigma
_{\alpha }^{\left( p,q\right) }\left( \beta \right) +\varepsilon \right)
\left( \log ^{\left[ q-1\right] }r\right) ^{\rho _{\alpha }^{\left(
p,q\right) }\left( \beta \right) }  \notag \\
i.e.,~\log ^{\left[ p-2\right] }\alpha ^{-1}\beta \left( r\right) & <\left[
\exp \left( \left( \log ^{\left[ q-1\right] }r\right) ^{\rho _{\alpha
}^{\left( p,q\right) }\left( \beta \right) }\right) \right] ^{\sigma
_{\alpha }^{\left( p,q\right) }\left( \beta \right) +\varepsilon }  \notag \\
i.e.,~\frac{\log ^{\left[ p-2\right] }\alpha ^{-1}\beta \left( r\right) }{%
\left[ \exp \left( \left( \log ^{\left[ q-1\right] }r\right) ^{\rho _{\alpha
}^{\left( p,q\right) }\left( \beta \right) }\right) \right] ^{k}}& <\frac{%
\left[ \exp \left( \left( \log ^{\left[ q-1\right] }r\right) ^{\rho _{\alpha
}^{\left( p,q\right) }\left( \beta \right) }\right) \right] ^{\sigma
_{\alpha }^{\left( p,q\right) }\left( \beta \right) +\varepsilon }}{\left[
\exp \left( \left( \log ^{\left[ q-1\right] }r\right) ^{\rho _{\alpha
}^{\left( p,q\right) }\left( \beta \right) }\right) \right] ^{k}}  \notag
\end{align}%
\begin{equation*}
i.e.,~\frac{\log ^{\left[ p-2\right] }\alpha ^{-1}\beta \left( r\right) }{%
\left[ \exp \left( \left( \log ^{\left[ q-1\right] }r\right) ^{\rho _{\alpha
}^{\left( p,q\right) }\left( \beta \right) }\right) \right] ^{k}}<~\ \ \ \ \
\ \ \ \ \ \ \ \ \ \ \ \ \ \ \ \ \ \ \ \ \ \ \ \ \ \ \ \ \ \ \ \ \ \
\end{equation*}%
\begin{equation*}
~\ \ \ \ \ \ \ \ \ \ \ \ \ \ \ \ \ \ \ \ \ \ \ \ \ \ \ \ \ \ \ \ \ \ \frac{1%
}{\left[ \exp \left( \left( \log ^{\left[ q-1\right] }r\right) ^{\rho
_{\alpha }^{\left( p,q\right) }\left( \beta \right) }\right) \right]
^{k-\left( \sigma _{\alpha }^{\left( p,q\right) }\left( \beta \right)
+\varepsilon \right) }}~.
\end{equation*}%
Therefore $\int\limits_{r_{0}}^{\infty }\frac{\log ^{\left[ p-2\right]
}\alpha ^{-1}\beta \left( r\right) }{\left[ \exp \left( \left( \log ^{\left[
q-1\right] }r\right) ^{\rho _{\alpha }^{\left( p,q\right) }\left( \beta
\right) }\right) \right] ^{k+1}}dr\left( r_{0}>0\right) $ converges for $%
k>\sigma _{\alpha }^{\left( p,q\right) }\left( \beta \right) .$

Again by Definition \ref{d1.4}, we obtain for a sequence values of $r$
tending to infinity that%
\begin{align}
\log ^{\left[ p-1\right] }\alpha ^{-1}\beta \left( r\right) & >\left( \sigma
_{\alpha }^{\left( p,q\right) }\left( \beta \right) -\varepsilon \right)
\left( \log ^{\left[ q-1\right] }r\right) ^{\rho _{\alpha }^{\left(
p,q\right) }\left( \beta \right) }  \notag \\
i.e.,~\log ^{\left[ p-2\right] }\alpha ^{-1}\beta \left( r\right) & >\left[
\exp \left( \left( \log ^{\left[ q-1\right] }r\right) ^{\rho _{\alpha
}^{\left( p,q\right) }\left( \beta \right) }\right) \right] ^{\sigma
_{\alpha }^{\left( p,q\right) }\left( \beta \right) -\varepsilon }~.
\label{3.4}
\end{align}%
So for $k<\sigma _{\alpha }^{\left( p,q\right) }\left( \beta \right) $, we
get from $\left( \ref{3.4}\right) $ that%
\begin{equation*}
\frac{\log ^{\left[ p-2\right] }\alpha ^{-1}\beta \left( r\right) }{\left[
\exp \left( \left( \log ^{\left[ q-1\right] }r\right) ^{\rho _{\alpha
}^{\left( p,q\right) }\left( \beta \right) }\right) \right] ^{k}}>\frac{1}{%
\left[ \exp \left( \left( \log ^{\left[ q-1\right] }r\right) ^{\rho _{\alpha
}^{\left( p,q\right) }\left( \beta \right) }\right) \right] ^{k-\left(
\sigma _{\alpha }^{\left( p,q\right) }\left( \beta \right) -\varepsilon
\right) }}~.
\end{equation*}%
Therefore $\int\limits_{r_{0}}^{\infty }\frac{\log ^{\left[ p-2\right]
}\alpha ^{-1}\beta \left( r\right) }{\left[ \exp \left( \left( \log ^{\left[
q-1\right] }r\right) ^{\rho _{\alpha }^{\left( p,q\right) }\left( \beta
\right) }\right) \right] ^{k+1}}dr\left( r_{0}>0\right) $ diverges for $%
k<\sigma _{\alpha }^{\left( p,q\right) }\left( \beta \right) $.\newline
Hence $\int\limits_{r_{0}}^{\infty }\frac{\log ^{\left[ p-2\right] }\alpha
^{-1}\beta \left( r\right) }{\left[ \exp \left( \left( \log ^{\left[ q-1%
\right] }r\right) ^{\rho _{\alpha }^{\left( p,q\right) }\left( \beta \right)
}\right) \right] ^{k+1}}dr\left( r_{0}>0\right) $ converges for $k>\sigma
_{\alpha }^{\left( p,q\right) }\left( \beta \right) $ and diverges for $%
k<\sigma _{\alpha }^{\left( p,q\right) }\left( \beta \right) $.\newline
\textbf{Subcase (B). }$\sigma _{\alpha }^{\left( p,q\right) }\left( \beta
\right) =0.$\newline

When $\sigma _{\alpha }^{\left( p,q\right) }\left( \beta \right) =0$ for any
two positive integers $p$\ and $q$ , Definition \ref{d1.4} gives for all
sufficiently large values of $r$ that%
\begin{equation*}
\frac{\log ^{\left[ p-1\right] }\alpha ^{-1}\beta \left( r\right) }{\left(
\log ^{\left[ q-1\right] }r\right) ^{\rho _{\alpha }^{\left( p,q\right)
}\left( \beta \right) }}<\varepsilon ~.
\end{equation*}%
Then as before we obtain that $\int\limits_{r_{0}}^{\infty }\frac{\log ^{%
\left[ p-2\right] }\alpha ^{-1}\beta \left( r\right) }{\left[ \exp \left(
\left( \log ^{\left[ q-1\right] }r\right) ^{\rho _{\alpha }^{\left(
p,q\right) }\left( \beta \right) }\right) \right] ^{k+1}}dr\left(
r_{0}>0\right) $ converges for $k>0$ and diverges for $k<0$.

Thus combining Subcase $\left( A\right) $ and Subcase $\left( B\right) $,
Definition \ref{d1.5} follows.\newline
\textbf{Definition \ref{d1.5} }$\Rightarrow$\textbf{\ Definition \ref{d1.4}.}%
\newline

From Definition \textbf{\ref{d1.5}} and for arbitrary positive $\varepsilon $
the integral \linebreak $\int\limits_{r_{0}}^{\infty }\frac{\log ^{\left[ p-2%
\right] }\alpha ^{-1}\beta \left( r\right) }{\left[ \exp \left( \left( \log
^{\left[ q-1\right] }r\right) ^{\rho _{\alpha }^{\left( p,q\right) }\left(
\beta \right) }\right) \right] ^{\sigma _{\alpha }^{\left( p,q\right)
}\left( \beta \right) +\varepsilon +1}}dr\left( r_{0}>0\right) $ converges.
Then by Lemma \ref{l1.1}, we get that%
\begin{equation*}
\underset{r\rightarrow \infty }{\lim \sup }\frac{\log ^{\left[ p-2\right]
}\alpha ^{-1}\beta \left( r\right) }{\left[ \exp \left( \left( \log ^{\left[
q-1\right] }r\right) ^{\rho _{\alpha }^{\left( p,q\right) }\left( \beta
\right) }\right) \right] ^{\sigma _{\alpha }^{\left( p,q\right) }\left(
\beta \right) +\varepsilon }}=0~.
\end{equation*}%
So we obtain all sufficiently large values of $r$ that%
\begin{align*}
\frac{\log ^{\left[ p-2\right] }\alpha ^{-1}\beta \left( r\right) }{\left[
\exp \left( \left( \log ^{\left[ q-1\right] }r\right) ^{\rho _{\alpha
}^{\left( p,q\right) }\left( \beta \right) }\right) \right] ^{\sigma
_{\alpha }^{\left( p,q\right) }\left( \beta \right) +\varepsilon }}&
<\varepsilon  \\
i.e.,~\log ^{\left[ p-2\right] }\alpha ^{-1}\beta \left( r\right) &
<\varepsilon \cdot \left[ \exp \left( \left( \log ^{\left[ q-1\right]
}r\right) ^{\rho _{\alpha }^{\left( p,q\right) }\left( \beta \right)
}\right) \right] ^{\sigma _{\alpha }^{\left( p,q\right) }\left( \beta
\right) +\varepsilon } \\
i.e.,~\log ^{\left[ p-1\right] }\alpha ^{-1}\beta \left( r\right) & <\log
\varepsilon +\left( \sigma _{\alpha }^{\left( p,q\right) }\left( \beta
\right) +\varepsilon \right) \left( \log ^{\left[ q-1\right] }r\right)
^{\rho _{\alpha }^{\left( p,q\right) }\left( \beta \right) } \\
i.e.,~\underset{r\rightarrow \infty }{\lim \sup }\frac{\log ^{\left[ p-1%
\right] }\alpha ^{-1}\beta \left( r\right) }{\left( \log ^{\left[ q-1\right]
}r\right) ^{\rho _{\alpha }^{\left( p,q\right) }\left( \beta \right) }}&
\leq \sigma _{\alpha }^{\left( p,q\right) }\left( \beta \right) +\varepsilon
~.
\end{align*}%
Since $\varepsilon \left( >0\right) $ is arbitrary, it follows from above
that%
\begin{equation}
\underset{r\rightarrow \infty }{\lim \sup }\frac{\log ^{\left[ p-1\right]
}\alpha ^{-1}\beta \left( r\right) }{\left( \log ^{\left[ q-1\right]
}r\right) ^{\rho _{\alpha }^{\left( p,q\right) }\left( \beta \right) }}\leq
\sigma _{\alpha }^{\left( p,q\right) }\left( \beta \right) ~.  \label{3.5}
\end{equation}%
On the other hand the divergence of the integral $\int\limits_{r_{0}}^{%
\infty }\frac{\log ^{\left[ p-2\right] }\alpha ^{-1}\beta \left( r\right) }{%
\left[ \exp \left( \left( \log ^{\left[ q-1\right] }r\right) ^{\rho _{\alpha
}^{\left( p,q\right) }\left( \beta \right) }\right) \right] ^{\sigma
_{\alpha }^{\left( p,q\right) }\left( \beta \right) -\varepsilon +1}}%
dr\left( r_{0}>0\right) $ implies that there exists a sequence of values of $%
r$ tending to infinity such that%
\begin{equation*}
\frac{\log ^{\left[ p-2\right] }\alpha ^{-1}\beta \left( r\right) }{\left[
\exp \left( \left( \log ^{\left[ q-1\right] }r\right) ^{\rho _{\alpha
}^{\left( p,q\right) }\left( \beta \right) }\right) \right] ^{\sigma
_{\alpha }^{\left( p,q\right) }\left( \beta \right) -\varepsilon +1}}>\frac{1%
}{\left[ \exp \left( \left( \log ^{\left[ q-1\right] }r\right) ^{\rho
_{\alpha }^{\left( p,q\right) }\left( \beta \right) }\right) \right]
^{1+\varepsilon }}
\end{equation*}%
\begin{align*}
i.e.,~\log ^{\left[ p-2\right] }\alpha ^{-1}\beta \left( r\right) & >\left[
\exp \left( \left( \log ^{\left[ q-1\right] }r\right) ^{\rho _{\alpha
}^{\left( p,q\right) }\left( \beta \right) }\right) \right] ^{\sigma
_{\alpha }^{\left( p,q\right) }\left( \beta \right) -2\varepsilon } \\
i.e.,~\log ^{\left[ p-1\right] }\alpha ^{-1}\beta \left( r\right) & >\left(
\sigma _{\alpha }^{\left( p,q\right) }\left( \beta \right) -2\varepsilon
\right) \left( \left( \log ^{\left[ q-1\right] }r\right) ^{\rho _{\alpha
}^{\left( p,q\right) }\left( \beta \right) }\right)  \\
i.e.,~\frac{\log ^{\left[ p-1\right] }\alpha ^{-1}\beta \left( r\right) }{%
\left( \log ^{\left[ q-1\right] }r\right) ^{\rho _{\alpha }^{\left(
p,q\right) }\left( \beta \right) }}& >\left( \sigma _{\alpha }^{\left(
p,q\right) }\left( \beta \right) -2\varepsilon \right) ~.
\end{align*}%
As $\varepsilon \left( >0\right) $ is arbitrary, it follows from above that%
\begin{equation}
\underset{r\rightarrow \infty }{\lim \sup }\frac{\log ^{\left[ p-1\right]
}\alpha ^{-1}\beta \left( r\right) }{\left( \log ^{\left[ q-1\right]
}r\right) ^{\rho _{\alpha }^{\left( p,q\right) }\left( \beta \right) }}\geq
\sigma _{\alpha }^{\left( p,q\right) }\left( \beta \right) ~.  \label{3.6}
\end{equation}%
So from $\left( \ref{3.5}\right) $ and $\left( \ref{3.6}\right) $ , we
obtain that%
\begin{equation*}
\underset{r\rightarrow \infty }{\lim \sup }\frac{\log ^{\left[ p-1\right]
}\alpha ^{-1}\beta \left( r\right) }{\left( \log ^{\left[ q-1\right]
}r\right) ^{\rho _{\alpha }^{\left( p,q\right) }\left( \beta \right) }}%
=\sigma _{\alpha }^{\left( p,q\right) }\left( \beta \right) ~.
\end{equation*}%
\newline
This proves the theorem.
\end{proof}

\begin{theorem}
\label{t3.2} Let $\alpha \left( x\right) $ and $\beta \left( x\right) $ be
any two positive continuous increasing to $+\infty $ on $[x_{0},+\infty )$
functions having finite positive \emph{relative }$\left( p,q\right) $\emph{\
-th lower order} $\lambda _{\alpha }^{\left( p,q\right) }\left( \beta
\right) $ $\left( 0<\lambda _{\alpha }^{\left( p,q\right) }\left( \beta
\right) <\infty \right) $ and \emph{relative }$\left( p,q\right) $\emph{\
-th weak type} $\tau _{\alpha }^{\left( p,q\right) }\left( \beta \right) $
where $p$\ and $q$ are any two positive integers. Then Definition \ref{d1.6}
and Definition \ref{d1.7} are equivalent.
\end{theorem}

\begin{proof}
Let us consider $\alpha \left( x\right) $ and $\beta \left( x\right) $ be
any two positive continuous increasing to $+\infty $ on $[x_{0},+\infty )$
functions such that $\lambda _{\alpha }^{\left( p,q\right) }\left( \beta
\right) $ $\left( 0<\lambda _{\alpha }^{\left( p,q\right) }\left( \beta
\right) <\infty \right) $ exists for any two positive integers $p$\ and $q$.%
\newline
\textbf{Case I.\ }$\tau _{\alpha }^{\left( p,q\right) }\left( \beta \right)
=\infty .$\newline
\textbf{Definition \ref{d1.6} }$\Rightarrow $\textbf{\ Definition \ref{d1.7}.%
}\newline
\qquad As $\tau _{\alpha }^{\left( p,q\right) }\left( \beta \right) =\infty $%
, from Definition \ref{d1.6} we obtain for arbitrary positive $G$ and for
all sufficiently large values of $r$ that%
\begin{align}
\log ^{\left[ p-1\right] }\alpha ^{-1}\beta \left( r\right) & >G\cdot \left(
\log ^{\left[ q-1\right] }r\right) ^{\lambda _{\alpha }^{\left( p,q\right)
}\left( \beta \right) }  \notag \\
i.e.,~\log ^{\left[ p-2\right] }\alpha ^{-1}\beta \left( r\right) & >\left[
\exp \left( \left( \log ^{\left[ q-1\right] }r\right) ^{\lambda _{\alpha
}^{\left( p,q\right) }\left( \beta \right) }\right) \right] ^{G}~.
\label{33.1}
\end{align}%
Now if possible let the integral $\int\limits_{r_{0}}^{\infty }\frac{\log ^{%
\left[ p-2\right] }\alpha ^{-1}\beta \left( r\right) }{\left[ \exp \left(
\left( \log ^{\left[ q-1\right] }r\right) ^{\lambda _{\alpha }^{\left(
p,q\right) }\left( \beta \right) }\right) \right] ^{G+1}}dr$ $\left(
r_{0}>0\right) $ be converge.\newline
Then by Lemma \ref{l1.1},%
\begin{equation*}
\underset{r\rightarrow \infty }{\lim \inf }\frac{\log ^{\left[ p-2\right]
}\alpha ^{-1}\beta \left( r\right) }{\left[ \exp \left( \left( \log ^{\left[
q-1\right] }r\right) ^{\lambda _{\alpha }^{\left( p,q\right) }\left( \beta
\right) }\right) \right] ^{G}}=0~.
\end{equation*}%
So for a sequence of values of $r$ tending to infinity we get that%
\begin{equation}
\log ^{\left[ p-2\right] }\alpha ^{-1}\beta \left( r\right) <\left[ \exp
\left( \left( \log ^{\left[ q-1\right] }r\right) ^{\lambda _{\alpha
}^{\left( p,q\right) }\left( \beta \right) }\right) \right] ^{G}~.
\label{33.2}
\end{equation}%
Therefore from $\left( \ref{33.1}\right) $ and $\left( \ref{33.2}\right) $,
we arrive at a contradiction.

Hence $\int\limits_{r_{0}}^{\infty }\frac{\log ^{\left[ p-2\right] }\alpha
^{-1}\beta \left( r\right) }{\left[ \exp \left( \left( \log ^{\left[ q-1%
\right] }r\right) ^{\lambda _{\alpha }^{\left( p,q\right) }\left( \beta
\right) }\right) \right] ^{G+1}}dr$ $\left( r_{0}>0\right) $ diverges
whenever $G$ is finite, which is the Definition \ref{d1.7}.\newline
\textbf{Definition \ref{d1.7} }$\Rightarrow $\textbf{\ Definition \ref{d1.6}.%
}\newline
\qquad Let $G$ be any positive number. Since $\tau _{\alpha }^{\left(
p,q\right) }\left( \beta \right) =\infty $, from Definition \ref{d1.7}, the
divergence of the integral $\int\limits_{r_{0}}^{\infty }\frac{\log ^{\left[
p-2\right] }\alpha ^{-1}\beta \left( r\right) }{\left[ \exp \left( \left(
\log ^{\left[ q-1\right] }r\right) ^{\lambda _{\alpha }^{\left( p,q\right)
}\left( \beta \right) }\right) \right] ^{G+1}}dr$ $\left( r_{0}>0\right) $
gives for arbitrary positive $\varepsilon $ and for all sufficiently large
values of $r$ that%
\begin{align*}
\log ^{\left[ p-2\right] }\alpha ^{-1}\beta \left( r\right) & >\left[ \exp
\left( \left( \log ^{\left[ q-1\right] }r\right) ^{\lambda _{\alpha
}^{\left( p,q\right) }\left( \beta \right) }\right) \right] ^{G-\varepsilon }
\\
i.e.,~\log ^{\left[ p-1\right] }\alpha ^{-1}\beta \left( r\right) & >\left(
G-\varepsilon \right) \left( \log ^{\left[ q-1\right] }r\right) ^{\lambda
_{\alpha }^{\left( p,q\right) }\left( \beta \right) },
\end{align*}%
which implies that%
\begin{equation*}
\underset{r\rightarrow \infty }{\lim \inf }\frac{\log ^{\left[ p-1\right]
}\alpha ^{-1}\beta \left( r\right) }{\left( \log ^{\left[ q-1\right]
}r\right) ^{\lambda _{\alpha }^{\left( p,q\right) }\left( \beta \right) }}%
\geq G-\varepsilon ~.
\end{equation*}%
Since $G>0$ is arbitrary, it follows that%
\begin{equation*}
\underset{r\rightarrow \infty }{\lim \inf }\frac{\log ^{\left[ p-1\right]
}\alpha ^{-1}\beta \left( r\right) }{\left( \log ^{\left[ q-1\right]
}r\right) ^{\lambda _{\alpha }^{\left( p,q\right) }\left( \beta \right) }}%
=\infty ~.
\end{equation*}%
Thus Definition \ref{d1.6} follows.\newline
\textbf{Case II.\ }$0\leq \tau _{\alpha }^{\left( p,q\right) }\left( \beta
\right) <\infty .$\newline
\textbf{Definition \ref{d1.6} }$\Rightarrow $\textbf{\ Definition \ref{d1.7}.%
}\newline
\textbf{Subcase (C).$\mathbb{\ }$\ }$0<\tau _{\alpha }^{\left( p,q\right)
}\left( \beta \right) <\infty .$\newline

Let $\alpha \left( x\right) $ and $\beta \left( x\right) $ be any two
positive continuous increasing to $+\infty $ on $[x_{0},+\infty )$ functions
such that $0<\tau _{\alpha }^{\left( p,q\right) }\left( \beta \right)
<\infty $ exists for any two positive integers $p$\ and $q.$ Then according
to the Definition \ref{d.16}, for a sequence of values of $r$ tending to
infinity we get that%
\begin{align}
\log ^{\left[ p-1\right] }\alpha ^{-1}\beta \left( r\right) & <\left( \tau
_{\alpha }^{\left( p,q\right) }\left( \beta \right) +\varepsilon \right)
\left( \log ^{\left[ q-1\right] }r\right) ^{\lambda _{\alpha }^{\left(
p,q\right) }\left( \beta \right) }  \notag \\
i.e.,~\log ^{\left[ p-2\right] }\alpha ^{-1}\beta \left( r\right) & <\left[
\exp \left( \left( \log ^{\left[ q-1\right] }r\right) ^{\lambda _{\alpha
}^{\left( p,q\right) }\left( \beta \right) }\right) \right] ^{\tau _{\alpha
}^{\left( p,q\right) }\left( \beta \right) +\varepsilon }  \notag \\
i.e.,~\frac{\log ^{\left[ p-2\right] }\alpha ^{-1}\beta \left( r\right) }{%
\left[ \exp \left( \left( \log ^{\left[ q-1\right] }r\right) ^{\lambda
_{\alpha }^{\left( p,q\right) }\left( \beta \right) }\right) \right] ^{k}}& <%
\frac{\left[ \exp \left( \left( \log ^{\left[ q-1\right] }r\right) ^{\lambda
_{\alpha }^{\left( p,q\right) }\left( \beta \right) }\right) \right] ^{\tau
_{\alpha }^{\left( p,q\right) }\left( \beta \right) +\varepsilon }}{\left[
\exp \left( \left( \log ^{\left[ q-1\right] }r\right) ^{\lambda _{\alpha
}^{\left( p,q\right) }\left( \beta \right) }\right) \right] ^{k}}  \notag
\end{align}%
\begin{equation*}
i.e.,~\frac{\log ^{\left[ p-2\right] }\alpha ^{-1}\beta \left( r\right) }{%
\left[ \exp \left( \left( \log ^{\left[ q-1\right] }r\right) ^{\lambda
_{\alpha }^{\left( p,q\right) }\left( \beta \right) }\right) \right] ^{k}}<%
\frac{1}{\left[ \exp \left( \left( \log ^{\left[ q-1\right] }r\right)
^{\lambda _{\alpha }^{\left( p,q\right) }\left( \beta \right) }\right) %
\right] ^{k-\left( \tau _{\alpha }^{\left( p,q\right) }\left( \beta \right)
+\varepsilon \right) }}~.
\end{equation*}%
Therefore $\int\limits_{r_{0}}^{\infty }\frac{\log ^{\left[ p-2\right]
}\alpha ^{-1}\beta \left( r\right) }{\left[ \exp \left( \left( \log ^{\left[
q-1\right] }r\right) ^{\lambda _{\alpha }^{\left( p,q\right) }\left( \beta
\right) }\right) \right] ^{k+1}}dr\left( r_{0}>0\right) $ converges for $%
k>\tau _{\alpha }^{\left( p,q\right) }\left( \beta \right) .$

Again by Definition \ref{d1.6}, we obtain for all sufficiently large values
of $r$ that%
\begin{align}
\log ^{\left[ p-1\right] }\alpha ^{-1}\beta \left( r\right) & >\left( \tau
_{\alpha }^{\left( p,q\right) }\left( \beta \right) -\varepsilon \right)
\left( \log ^{\left[ q-1\right] }r\right) ^{\lambda _{\alpha }^{\left(
p,q\right) }\left( \beta \right) }  \notag \\
i.e.,~\log ^{\left[ p-2\right] }\alpha ^{-1}\beta \left( r\right) & >\left[
\exp \left( \left( \log ^{\left[ q-1\right] }r\right) ^{\lambda _{\alpha
}^{\left( p,q\right) }\left( \beta \right) }\right) \right] ^{\tau _{\alpha
}^{\left( p,q\right) }\left( \beta \right) -\varepsilon }~.  \label{33.4}
\end{align}%
So for $k<\tau _{\alpha }^{\left( p,q\right) }\left( \beta \right) $, we get
from $\left( \ref{33.4}\right) $ that%
\begin{equation*}
\frac{\log ^{\left[ p-2\right] }\alpha ^{-1}\beta \left( r\right) }{\left[
\exp \left( \left( \log ^{\left[ q-1\right] }r\right) ^{\lambda _{\alpha
}^{\left( p,q\right) }\left( \beta \right) }\right) \right] ^{k}}>\frac{1}{%
\left[ \exp \left( \left( \log ^{\left[ q-1\right] }r\right) ^{\lambda
_{\alpha }^{\left( p,q\right) }\left( \beta \right) }\right) \right]
^{k-\left( \tau _{\alpha }^{\left( p,q\right) }\left( \beta \right)
-\varepsilon \right) }}~.
\end{equation*}%
Therefore $\int\limits_{r_{0}}^{\infty }\frac{\log ^{\left[ p-2\right]
}\alpha ^{-1}\beta \left( r\right) }{\left[ \exp \left( \left( \log ^{\left[
q-1\right] }r\right) ^{\lambda _{\alpha }^{\left( p,q\right) }\left( \beta
\right) }\right) \right] ^{k+1}}dr\left( r_{0}>0\right) $ diverges for $%
k<\tau _{\alpha }^{\left( p,q\right) }\left( \beta \right) $.\newline
Hence $\int\limits_{r_{0}}^{\infty }\frac{\log ^{\left[ p-2\right] }\alpha
^{-1}\beta \left( r\right) }{\left[ \exp \left( \left( \log ^{\left[ q-1%
\right] }r\right) ^{\lambda _{\alpha }^{\left( p,q\right) }\left( \beta
\right) }\right) \right] ^{k+1}}dr\left( r_{0}>0\right) $ converges for $%
k>\tau _{\alpha }^{\left( p,q\right) }\left( \beta \right) $ and diverges
for $k<\tau _{\alpha }^{\left( p,q\right) }\left( \beta \right) $.\newline
\textbf{Subcase (D). }$\tau _{\alpha }^{\left( p,q\right) }\left( \beta
\right) =0.$\newline

When $\tau _{\alpha }^{\left( p,q\right) }\left( \beta \right) =0$ for any
two positive integers $p$\ and $q$ , Definition \ref{d1.6} gives for a
sequence of values of $r$ tending to infinity that%
\begin{equation*}
\frac{\log ^{\left[ p-1\right] }\alpha ^{-1}\beta \left( r\right) }{\left(
\log ^{\left[ q-1\right] }r\right) ^{\lambda _{\alpha }^{\left( p,q\right)
}\left( \beta \right) }}<\varepsilon ~.
\end{equation*}%
Then as before we obtain that $\int\limits_{r_{0}}^{\infty }\frac{\log ^{%
\left[ p-2\right] }\alpha ^{-1}\beta \left( r\right) }{\left[ \exp \left(
\left( \log ^{\left[ q-1\right] }r\right) ^{\lambda _{\alpha }^{\left(
p,q\right) }\left( \beta \right) }\right) \right] ^{k+1}}dr\left(
r_{0}>0\right) $ converges for $k>0$ and diverges for $k<0$.

Thus combining Subcase $\left( C\right) $ and Subcase $\left( D\right) $,
Definition \ref{d1.7} follows.\newline
\textbf{Definition \ref{d1.7} }$\Rightarrow $\textbf{\ Definition \ref{d1.6}.%
} %\newline

From Definition \textbf{\ref{d1.7}} and for arbitrary positive $\varepsilon $
the integral \linebreak $\int\limits_{r_{0}}^{\infty }\frac{\log ^{\left[ p-2%
\right] }\alpha ^{-1}\beta \left( r\right) }{\left[ \exp \left( \left( \log
^{\left[ q-1\right] }r\right) ^{\lambda _{\alpha }^{\left( p,q\right)
}\left( \beta \right) }\right) \right] ^{\tau _{\alpha }^{\left( p,q\right)
}\left( \beta \right) +\varepsilon +1}}dr\left( r_{0}>0\right) $ converges.
Then by Lemma \ref{l1.1}, we get that%
\begin{equation*}
\underset{r\rightarrow \infty }{\lim \inf }\frac{\log ^{\left[ p-2\right]
}\alpha ^{-1}\beta \left( r\right) }{\left[ \exp \left( \left( \log ^{\left[
q-1\right] }r\right) ^{\lambda _{\alpha }^{\left( p,q\right) }\left( \beta
\right) }\right) \right] ^{\tau _{\alpha }^{\left( p,q\right) }\left( \beta
\right) +\varepsilon }}=0~.
\end{equation*}%
So we get for a sequence of values of $r$ tending to infinity that%
\begin{align*}
\frac{\log ^{\left[ p-2\right] }\alpha ^{-1}\beta \left( r\right) }{\left[
\exp \left( \left( \log ^{\left[ q-1\right] }r\right) ^{\lambda _{\alpha
}^{\left( p,q\right) }\left( \beta \right) }\right) \right] ^{\tau _{\alpha
}^{\left( p,q\right) }\left( \beta \right) +\varepsilon }}& <\varepsilon  \\
i.e.,~\log ^{\left[ p-2\right] }\alpha ^{-1}\beta \left( r\right) &
<\varepsilon \cdot \left[ \exp \left( \left( \log ^{\left[ q-1\right]
}r\right) ^{\lambda _{\alpha }^{\left( p,q\right) }\left( \beta \right)
}\right) \right] ^{\tau _{\alpha }^{\left( p,q\right) }\left( \beta \right)
+\varepsilon } \\
i.e.,~\log ^{\left[ p-1\right] }\alpha ^{-1}\beta \left( r\right) & <\log
\varepsilon +\left( \tau _{\alpha }^{\left( p,q\right) }\left( \beta \right)
+\varepsilon \right) \left( \log ^{\left[ q-1\right] }r\right) ^{\lambda
_{\alpha }^{\left( p,q\right) }\left( \beta \right) } \\
i.e.,~\underset{r\rightarrow \infty }{\lim \inf }\frac{\log ^{\left[ p-1%
\right] }\alpha ^{-1}\beta \left( r\right) }{\left( \log ^{\left[ q-1\right]
}r\right) ^{\lambda _{\alpha }^{\left( p,q\right) }\left( \beta \right) }}&
\leq \tau _{\alpha }^{\left( p,q\right) }\left( \beta \right) +\varepsilon ~.
\end{align*}%
Since $\varepsilon \left( >0\right) $ is arbitrary, it follows from above
that%
\begin{equation}
\underset{r\rightarrow \infty }{\lim \inf }\frac{\log ^{\left[ p-1\right]
}\alpha ^{-1}\beta \left( r\right) }{\left( \log ^{\left[ q-1\right]
}r\right) ^{\lambda _{\alpha }^{\left( p,q\right) }\left( \beta \right) }}%
\leq \tau _{\alpha }^{\left( p,q\right) }\left( \beta \right) ~.
\label{33.5}
\end{equation}%
On the other hand the divergence of the integral $\int\limits_{r_{0}}^{%
\infty }\frac{\log ^{\left[ p-2\right] }\alpha ^{-1}\beta \left( r\right) }{%
\left[ \exp \left( \left( \log ^{\left[ q-1\right] }r\right) ^{\lambda
_{\alpha }^{\left( p,q\right) }\left( \beta \right) }\right) \right] ^{\tau
_{\alpha }^{\left( p,q\right) }\left( \beta \right) -\varepsilon +1}}%
dr\left( r_{0}>0\right) $ implies for all sufficiently large values of $r$
that%
\begin{equation*}
\frac{\log ^{\left[ p-2\right] }\alpha ^{-1}\beta \left( r\right) }{\left[
\exp \left( \left( \log ^{\left[ q-1\right] }r\right) ^{\lambda _{\alpha
}^{\left( p,q\right) }\left( \beta \right) }\right) \right] ^{\tau _{\alpha
}^{\left( p,q\right) }\left( \beta \right) -\varepsilon +1}}>\frac{1}{\left[
\exp \left( \left( \log ^{\left[ q-1\right] }r\right) ^{\lambda _{\alpha
}^{\left( p,q\right) }\left( \beta \right) }\right) \right] ^{1+\varepsilon }%
}
\end{equation*}%
\begin{align*}
i.e.,~\log ^{\left[ p-2\right] }\alpha ^{-1}\beta \left( r\right) & >\left[
\exp \left( \left( \log ^{\left[ q-1\right] }r\right) ^{\lambda _{\alpha
}^{\left( p,q\right) }\left( \beta \right) }\right) \right] ^{\tau _{\alpha
}^{\left( p,q\right) }\left( \beta \right) -2\varepsilon } \\
i.e.,~\log ^{\left[ p-1\right] }\alpha ^{-1}\beta \left( r\right) & >\left(
\tau _{\alpha }^{\left( p,q\right) }\left( \beta \right) -2\varepsilon
\right) \left( \left( \log ^{\left[ q-1\right] }r\right) ^{\lambda _{\alpha
}^{\left( p,q\right) }\left( \beta \right) }\right)  \\
i.e.,~\frac{\log ^{\left[ p-1\right] }\alpha ^{-1}\beta \left( r\right) }{%
\left( \log ^{\left[ q-1\right] }r\right) ^{\lambda _{\alpha }^{\left(
p,q\right) }\left( \beta \right) }}& >\left( \tau _{\alpha }^{\left(
p,q\right) }\left( \beta \right) -2\varepsilon \right) ~.
\end{align*}%
As $\varepsilon \left( >0\right) $ is arbitrary, it follows from above that%
\begin{equation}
\underset{r\rightarrow \infty }{\lim \inf }\frac{\log ^{\left[ p-1\right]
}\alpha ^{-1}\beta \left( r\right) }{\left( \log ^{\left[ q-1\right]
}r\right) ^{\lambda _{\alpha }^{\left( p,q\right) }\left( \beta \right) }}%
\geq \tau _{\alpha }^{\left( p,q\right) }\left( \beta \right) ~.
\label{33.6}
\end{equation}%
So from $\left( \ref{33.5}\right) $ and $\left( \ref{33.6}\right) $ we
obtain that%
\begin{equation*}
\underset{r\rightarrow \infty }{\lim \inf }\frac{\log ^{\left[ p-1\right]
}\alpha ^{-1}\beta \left( r\right) }{\left( \log ^{\left[ q-1\right]
}r\right) ^{\lambda _{\alpha }^{\left( p,q\right) }\left( \beta \right) }}%
=\tau _{\alpha }^{\left( p,q\right) }\left( \beta \right) ~.
\end{equation*}%
This proves the theorem.
\end{proof}

\qquad Next we introduce the following two relative growth indicators which
will also enable help our subsequent study.

\begin{definition}
\label{d1.8} Let $\alpha \left( x\right) $ and $\beta \left( x\right) $ be
any two positive continuous increasing to $+\infty $ on $[x_{0},+\infty )$
functions having finite positive relative $\left( p,q\right) $ th order $%
\rho _{\alpha }^{\left( p,q\right) }\left( \beta \right) $ \linebreak $%
\left( a<\rho _{\alpha }^{\left( p,q\right) }\left( \beta \right) <\infty
\right) $ where $p$\ and $q$ are any two positive integers. Then the \emph{%
relative }$\left( p,q\right) $\emph{-th lower type} of $\beta \left(
x\right) $ with respect to $\alpha \left( x\right) $ is defined as :%
\begin{equation*}
\overline{\sigma }_{\alpha }^{\left( p,q\right) }\left( \beta \right) =%
\underset{r\rightarrow \infty }{\lim \inf }\frac{\log ^{\left[ p-1\right]
}\alpha ^{-1}\beta \left( r\right) }{\left( \log ^{\left[ q-1\right]
}r\right) ^{\rho _{\alpha }^{\left( p,q\right) }\left( \beta \right) }}~.
\end{equation*}
\end{definition}

\qquad The above definition can alternatively be defined in the following
manner:

\begin{definition}
\label{d1.9} Let $\alpha \left( x\right) $ and $\beta \left( x\right) $ be
any two positive continuous increasing to $+\infty $ on $[x_{0},+\infty )$
functions having finite positive \emph{relative }$\left( p,q\right) $\emph{\
-th order} $\rho _{\alpha }^{\left( p,q\right) }\left( \beta \right) $ $%
\left( a<\rho _{\alpha }^{\left( p,q\right) }\left( \beta \right) <\infty
\right) $ where $p$\ and $q$ are any two positive integers. Then the \emph{%
relative }$\left( p,q\right) $\emph{\ -th lower type} $\overline{\sigma }%
_{\alpha }^{\left( p,q\right) }\left( \beta \right) $ of $\beta \left(
x\right) $ with respect to $\alpha \left( x\right) $ is defined as: The
integral \linebreak $\int\limits_{r_{0}}^{\infty }\frac{\log ^{\left[ p-2%
\right] }\alpha ^{-1}\beta \left( r\right) }{\left[ \exp \left( \left( \log
^{\left[ q-1\right] }r\right) ^{\rho _{\alpha }^{\left( p,q\right) }\left(
\beta \right) }\right) \right] ^{k+1}}dr\left( r_{0}>0\right) $ converges
for $k>\overline{\sigma }_{\alpha }^{\left( p,q\right) }\left( \beta \right)
$ and diverges for $k<\overline{\sigma }_{\alpha }^{\left( p,q\right)
}\left( \beta \right) .$
\end{definition}

\begin{definition}
\label{d1.10} Let $\alpha \left( x\right) $ and $\beta \left( x\right) $ be
any two positive continuous increasing to $+\infty $ on $[x_{0},+\infty )$
functions having finite positive relative $\left( p,q\right) $\emph{-th
lower order} $\lambda _{\alpha }^{\left( p,q\right) }\left( \beta \right) $
\linebreak $\left( a<\lambda _{\alpha }^{\left( p,q\right) }\left( \beta
\right) <\infty \right) $. Then the growth indicator $\overline{\tau }%
_{\alpha }^{\left( p,q\right) }\left( \beta \right) $ of $\beta \left(
x\right) $ with respect to $\alpha \left( x\right) $ is defined as :%
\begin{equation*}
\overline{\tau }_{\alpha }^{\left( p,q\right) }\left( \beta \right) =%
\underset{r\rightarrow \infty }{\lim \sup }\frac{\log ^{\left[ p-1\right]
}\alpha ^{-1}\beta \left( r\right) }{\left( \log ^{\left[ q-1\right]
}r\right) ^{\lambda _{\alpha }^{\left( p,q\right) }\left( \beta \right) }}~.
\end{equation*}
\end{definition}

\qquad The above definition can also alternatively defined as:

\begin{definition}
\label{d1.11} Let $\alpha \left( x\right) $ and $\beta \left( x\right) $ be
any two positive continuous increasing to $+\infty $ on $[x_{0},+\infty )$
functions having finite positive \emph{relative }$\left( p,q\right) $\emph{\
-th lower order} $\lambda _{\alpha }^{\left( p,q\right) }\left( \beta
\right) $ $\left( a<\lambda _{\alpha }^{\left( p,q\right) }\left( \beta
\right) <\infty \right) $ where $p$\ and $q$ are any two positive integers.
Then the growth indicator $\overline{\tau }_{\alpha }^{\left( p,q\right)
}\left( \beta \right) $ of $\beta \left( x\right) $ with respect to $\alpha
\left( x\right) $ is defined as: The integral \linebreak $%
\int\limits_{r_{0}}^{\infty }\frac{\log ^{\left[ p-2\right] }\alpha
^{-1}\beta \left( r\right) }{\left[ \exp \left( \left( \log ^{\left[ q-1%
\right] }r\right) ^{\lambda _{\alpha }^{\left( p,q\right) }\left( \beta
\right) }\right) \right] ^{k+1}}dr\left( r_{0}>0\right) $ converges for $k>%
\overline{\tau }_{g}^{\left( p,q\right) }\left( f\right) $ and diverges for $%
k<\overline{\tau }_{g}^{\left( p,q\right) }\left( f\right) .$
\end{definition}

\qquad Now we state the following two theorems without their proofs as those
can easily be carried out with help of Lemma \ref{l1.1} and in the line of
Theorem \ref{t3.1} and Theorem \ref{t3.2} respectively.

\begin{theorem}
\label{t3.4} Let $\alpha \left( x\right) $ and $\beta \left( x\right) $ be
any two positive continuous increasing to $+\infty $ on $[x_{0},+\infty )$
functions having finite positive \emph{relative }$\left( p,q\right) $\emph{\
-th order} $\rho _{\alpha }^{\left( p,q\right) }\left( \beta \right) $
\linebreak $\left( 0<\rho _{\alpha }^{\left( p,q\right) }\left( \beta
\right) <\infty \right) $ and \emph{relative }$\left( p,q\right) $\emph{\
-th lower type} $\overline{\sigma }_{\alpha }^{\left( p,q\right) }\left(
\beta \right) $ where $p$\ and $q$ are any two positive integers. Then
Definition \ref{d1.8} and Definition \ref{d1.9} are equivalent.
\end{theorem}

\begin{theorem}
\label{t3.5} $\alpha \left( x\right) $ and $\beta \left( x\right) $ be any
two positive continuous increasing to $+\infty $ on $[x_{0},+\infty )$
functions having finite positive \emph{relative }$\left( p,q\right) $\emph{\
-th lower order} $\lambda _{\alpha }^{\left( p,q\right) }\left( \beta
\right) $ \linebreak $\left( 0<\lambda _{\alpha }^{\left( p,q\right) }\left(
\beta \right) <\infty \right) $ and the growth indicator $\overline{\tau }%
_{\alpha }^{\left( p,q\right) }\left( \beta \right) $ where $p$\ and $q$ are
any two positive integers. Then Definition \ref{d1.10} and Definition \ref%
{d1.11} are equivalent.
\end{theorem}

\begin{theorem}
\label{t3.6} $\alpha \left( x\right) $ and $\beta \left( x\right) $ be any
two positive continuous increasing to $+\infty $ on $[x_{0},+\infty )$
functions with $0<\lambda _{\alpha }^{\left( p,q\right) }\left( \beta
\right) \leq \rho _{\alpha }^{\left( p,q\right) }\left( \beta \right)
<\infty $ where $p$\ and $q$ are any two positive integers. Then%
\begin{equation*}
\left( i\right) ~\sigma _{\alpha }^{\left( p,q\right) }\left( \beta \right) =%
\underset{r\rightarrow \infty }{\lim \sup }\frac{\log ^{\left[ p-1\right]
}\alpha ^{-1}\left( r\right) }{\left[ \log ^{\left[ q-1\right] }\beta
^{-1}\left( r\right) \right] ^{\rho _{\alpha }^{\left( p,q\right) }\left(
\beta \right) }},
\end{equation*}%
\begin{equation*}
\left( ii\right) ~\overline{\sigma }_{\alpha }^{\left( p,q\right) }\left(
\beta \right) =\underset{r\rightarrow \infty }{\lim \inf }\frac{\log ^{\left[
p-1\right] }\alpha ^{-1}\left( r\right) }{\left[ \log ^{\left[ q-1\right]
}\beta ^{-1}\left( r\right) \right] ^{\rho _{\alpha }^{\left( p,q\right)
}\left( \beta \right) }},
\end{equation*}%
\begin{equation*}
\left( iii\right) ~\tau _{\alpha }^{\left( p,q\right) }\left( \beta \right) =%
\underset{r\rightarrow \infty }{\lim \inf }\frac{\log ^{\left[ p-1\right]
}\alpha ^{-1}\left( r\right) }{\left[ \log ^{\left[ q-1\right] }\beta
^{-1}\left( r\right) \right] ^{\lambda _{\alpha }^{\left( p,q\right) }\left(
\beta \right) }}
\end{equation*}%
and%
\begin{equation*}
\left( iv\right) ~\overline{\tau }_{\alpha }^{\left( p,q\right) }\left(
\beta \right) =\underset{r\rightarrow \infty }{\lim \sup }\frac{\log ^{\left[
p-1\right] }\alpha ^{-1}\left( r\right) }{\left[ \log ^{\left[ q-1\right]
}\beta ^{-1}\left( r\right) \right] ^{\lambda _{\alpha }^{\left( p,q\right)
}\left( \beta \right) }}~.
\end{equation*}
\end{theorem}

\begin{proof}
Taking $\beta \left( r\right) =R,$ theorem follows from the definitions of $%
\sigma _{\alpha }^{\left( p,q\right) }\left( \beta \right) ,$ $\overline{%
\sigma }_{\alpha }^{\left( p,q\right) }\left( \beta \right) ,$ $\tau
_{\alpha }^{\left( p,q\right) }\left( \beta \right) $ and $\overline{\tau }%
_{\alpha }^{\left( p,q\right) }\left( \beta \right) $ respectively.
\end{proof}

\qquad In the following theorem we obtain a relationship between $\sigma
_{\alpha }^{\left( p,q\right) }\left( \beta \right) ,$ $\overline{\sigma }%
_{\alpha }^{\left( p,q\right) }\left( \beta \right) ,$ $\overline{\tau }%
_{\alpha }^{\left( p,q\right) }\left( \beta \right) $ and $\tau _{\alpha
}^{\left( p,q\right) }\left( \beta \right) $.

\begin{theorem}
\label{t3.7} Let $\alpha \left( x\right) $ and $\beta \left( x\right) $ be
any two positive continuous increasing to $+\infty $ on $[x_{0},+\infty )$
functions such $\rho _{\alpha }^{\left( p,q\right) }\left( \beta \right)
=\lambda _{\alpha }^{\left( p,q\right) }\left( \beta \right) $ $\left(
0<\lambda _{\alpha }^{\left( p,q\right) }\left( \beta \right) =\rho _{\alpha
}^{\left( p,q\right) }\left( \beta \right) <\infty \right) $ where $p$\ and $%
q$ are any two positive integers, then the following quantities%
\begin{equation*}
\left( i\right) \text{ }\sigma _{\alpha }^{\left( p,q\right) }\left( \beta
\right) ,\text{ }\left( ii\right) \text{ }\tau _{\alpha }^{\left( p,q\right)
}\left( \beta \right) ,\text{ }\left( iii\right) \text{ }\overline{\sigma }%
_{\alpha }^{\left( p,q\right) }\left( \beta \right) \text{\ and \ }\left(
iv\right) \text{ }\overline{\tau }_{\alpha }^{\left( p,q\right) }\left(
\beta \right)
\end{equation*}%
are all equivalent.
\end{theorem}

\begin{proof}
From Definition \ref{d1.7}, it follows that the integral $%
\int\limits_{r_{0}}^{\infty }\frac{\log ^{\left[ p-2\right] }\alpha
^{-1}\beta \left( r\right) }{\left[ \exp \left( \left( \log ^{\left[ q-1%
\right] }r\right) ^{\lambda _{\alpha }^{\left( p,q\right) }\left( \beta
\right) }\right) \right] ^{k+1}}dr\left( r_{0}>0\right) $ converges for $%
k>\tau _{\alpha }^{\left( p,q\right) }\left( \beta \right) $ and diverges
for $k<\tau _{\alpha }^{\left( p,q\right) }\left( \beta \right) $. On the
other hand, Definition \ref{d1.5} implies that the integral $%
\int\limits_{r_{0}}^{\infty }\frac{\log ^{\left[ p-2\right] }\alpha
^{-1}\beta \left( r\right) }{\left[ \exp \left( \left( \log ^{\left[ q-1%
\right] }r\right) ^{\rho _{\alpha }^{\left( p,q\right) }\left( \beta \right)
}\right) \right] ^{k+1}}dr\left( r_{0}>0\right) $ converges for $k>\sigma
_{\alpha }^{\left( p,q\right) }\left( \beta \right) $ and diverges for $%
k<\sigma _{\alpha }^{\left( p,q\right) }\left( \beta \right) $.\newline
$\left( i\right) \mathbf{\Rightarrow }\left( ii\right) $\textbf{.}

Now it is obvious that all the quantities in the expression%
\begin{equation*}
\left[ \frac{\log ^{\left[ p-2\right] }\alpha ^{-1}\beta \left( r\right) }{%
\left[ \exp \left( \left( \log ^{\left[ q-1\right] }r\right) ^{\lambda
_{\alpha }^{\left( p,q\right) }\left( \beta \right) }\right) \right] ^{k+1}}-%
\frac{\log ^{\left[ p-2\right] }\alpha ^{-1}\beta \left( r\right) }{\left[
\exp \left( \left( \log ^{\left[ q-1\right] }r\right) ^{\rho _{\alpha
}^{\left( p,q\right) }\left( \beta \right) }\right) \right] ^{k+1}}\right]
\end{equation*}%
are of non negative type. So%
\begin{multline*}
\int\limits_{r_{0}}^{\infty }\left[ \frac{\log ^{\left[ p-2\right] }\alpha
^{-1}\beta \left( r\right) }{\left[ \exp \left( \left( \log ^{\left[ q-1%
\right] }r\right) ^{\lambda _{\alpha }^{\left( p,q\right) }\left( \beta
\right) }\right) \right] ^{k+1}}\right.  \\
\left. -\frac{\log ^{\left[ p-2\right] }\alpha ^{-1}\beta \left( r\right) }{%
\left[ \exp \left( \left( \log ^{\left[ q-1\right] }r\right) ^{\rho _{\alpha
}^{\left( p,q\right) }\left( \beta \right) }\right) \right] ^{k+1}}\right]
dr\left( r_{0}>0\right) \geq 0
\end{multline*}%
\begin{multline*}
i.e.,~\int\limits_{r_{0}}^{\infty }\frac{\log ^{\left[ p-2\right] }\alpha
^{-1}\beta \left( r\right) }{\left[ \exp \left( \left( \log ^{\left[ q-1%
\right] }r\right) ^{\lambda _{\alpha }^{\left( p,q\right) }\left( \beta
\right) }\right) \right] ^{k+1}}dr\geq  \\
\int\limits_{r_{0}}^{\infty }\frac{\log ^{\left[ p-2\right] }\alpha
^{-1}\beta \left( r\right) }{\left[ \exp \left( \left( \log ^{\left[ q-1%
\right] }r\right) ^{\rho _{\alpha }^{\left( p,q\right) }\left( \beta \right)
}\right) \right] ^{k+1}}dr\text{ for }r_{0}>0~.
\end{multline*}%
\begin{equation}
i.e.,~\tau _{\alpha }^{\left( p,q\right) }\left( \beta \right) \geq \sigma
_{\alpha }^{\left( p,q\right) }\left( \beta \right) ~.  \label{3.8}
\end{equation}

\qquad Further as $\rho _{\alpha }^{\left( p,q\right) }\left( \beta \right)
=\lambda _{\alpha }^{\left( p,q\right) }\left( \beta \right) $, therefore we
get that%
\begin{align}
\sigma _{\alpha }^{\left( p,q\right) }\left( \beta \right) & =\text{ }%
\underset{r\rightarrow \infty }{\lim \sup }\frac{\log ^{\left[ p-1\right]
}\alpha ^{-1}\beta \left( r\right) }{\left( \log ^{\left[ q-1\right]
}r\right) ^{\rho _{\alpha }^{\left( p,q\right) }\left( \beta \right) }}
\notag \\
& \geq \text{ }\underset{r\rightarrow \infty }{\lim \inf }\frac{\log ^{\left[
p-1\right] }\alpha ^{-1}\beta \left( r\right) }{\left( \log ^{\left[ q-1%
\right] }r\right) ^{\rho _{\alpha }^{\left( p,q\right) }\left( \beta \right)
}}=\text{ }\underset{r\rightarrow \infty }{\lim \inf }\frac{\log ^{\left[ p-1%
\right] }\alpha ^{-1}\beta \left( r\right) }{\left( \log ^{\left[ q-1\right]
}r\right) ^{\lambda _{\alpha }^{\left( p,q\right) }\left( \beta \right) }}%
=\tau _{\alpha }^{\left( p,q\right) }\left( \beta \right) ~.  \label{3.7}
\end{align}

Hence from $\left( \ref{3.8}\right) $ and $\left( \ref{3.7}\right) $ we
obtain that%
\begin{equation}
\sigma _{\alpha }^{\left( p,q\right) }\left( \beta \right) =\tau _{\alpha
}^{\left( p,q\right) }\left( \beta \right) ~.  \label{3.9}
\end{equation}%
$\left( ii\right) \Rightarrow \left( iii\right) .$

Since $\rho _{\alpha }^{\left( p,q\right) }\left( \beta \right) =\lambda
_{\alpha }^{\left( p,q\right) }\left( \beta \right) $, we get that%
\begin{equation*}
\tau _{\alpha }^{\left( p,q\right) }\left( \beta \right) =\underset{%
r\rightarrow \infty }{\lim \inf }\frac{\log ^{\left[ p-1\right] }\alpha
^{-1}\beta \left( r\right) }{\left( \log ^{\left[ q-1\right] }r\right)
^{\lambda _{\alpha }^{\left( p,q\right) }\left( \beta \right) }}=\underset{%
r\rightarrow \infty }{\lim \inf }\frac{\log ^{\left[ p-1\right] }\alpha
^{-1}\beta \left( r\right) }{\left( \log ^{\left[ q-1\right] }r\right)
^{\rho _{\alpha }^{\left( p,q\right) }\left( \beta \right) }}=\overline{%
\sigma }_{\alpha }^{\left( p,q\right) }\left( \beta \right) ~.
\end{equation*}%
$\left( iii\right) \Rightarrow \left( iv\right) .$

In view of $\left( \ref{3.9}\right) $ and the condition $\rho _{\alpha
}^{\left( p,q\right) }\left( \beta \right) =\lambda _{\alpha }^{\left(
p,q\right) }\left( \beta \right) $, it follows that%
\begin{align*}
\overline{\sigma }_{\alpha }^{\left( p,q\right) }\left( \beta \right) & =%
\underset{r\rightarrow \infty }{\lim \inf }\frac{\log ^{\left[ p-1\right]
}\alpha ^{-1}\beta \left( r\right) }{\left( \log ^{\left[ q-1\right]
}r\right) ^{\rho _{\alpha }^{\left( p,q\right) }\left( \beta \right) }} \\
i.e.,~\overline{\sigma }_{\alpha }^{\left( p,q\right) }\left( \beta \right)
& =\underset{r\rightarrow \infty }{\lim \inf }\frac{\log ^{\left[ p-1\right]
}\alpha ^{-1}\beta \left( r\right) }{\left( \log ^{\left[ q-1\right]
}r\right) ^{\lambda _{\alpha }^{\left( p,q\right) }\left( \beta \right) }}
\end{align*}%
\begin{align*}
i.e.,~\overline{\sigma }_{\alpha }^{\left( p,q\right) }\left( \beta \right)
& =\tau _{\alpha }^{\left( p,q\right) }\left( \beta \right) ~\ \ \ \ \ \ \ \
\ \ \ \ \ \ \ \ \ \ \  \\
i.e.,~\overline{\sigma }_{\alpha }^{\left( p,q\right) }\left( \beta \right)
& =\sigma _{\alpha }^{\left( p,q\right) }\left( \beta \right)
\end{align*}%
\begin{align*}
i.e.,~\overline{\sigma }_{\alpha }^{\left( p,q\right) }\left( \beta \right)
& =\underset{r\rightarrow \infty }{\lim \sup }\frac{\log ^{\left[ p-1\right]
}\alpha ^{-1}\beta \left( r\right) }{\left( \log ^{\left[ q-1\right]
}r\right) ^{\rho _{\alpha }^{\left( p,q\right) }\left( \beta \right) }} \\
i.e.,~\overline{\sigma }_{\alpha }^{\left( p,q\right) }\left( \beta \right)
& =\underset{r\rightarrow \infty }{\lim \sup }\frac{\log ^{\left[ p-1\right]
}\alpha ^{-1}\beta \left( r\right) }{\left( \log ^{\left[ q-1\right]
}r\right) ^{\lambda _{\alpha }^{\left( p,q\right) }\left( \beta \right) }} \\
i.e.,~\overline{\sigma }_{\alpha }^{\left( p,q\right) }\left( \beta \right)
& =\overline{\tau }_{\alpha }^{\left( p,q\right) }\left( \beta \right) ~.
\end{align*}%
$\left( iv\right) \Rightarrow \left( i\right) .$

As $\rho _{\alpha }^{\left( p,q\right) }\left( \beta \right) =\lambda
_{\alpha }^{\left( p,q\right) }\left( \beta \right) ,$ we obtain that%
\begin{equation*}
\underset{r\rightarrow \infty }{\lim \sup }\frac{\log ^{\left[ p-1\right]
}\alpha ^{-1}\beta \left( r\right) }{\left( \log ^{\left[ q-1\right]
}r\right) ^{\lambda _{\alpha }^{\left( p,q\right) }\left( \beta \right) }}=%
\underset{r\rightarrow \infty }{\lim \sup }\frac{\log ^{\left[ p-1\right]
}\alpha ^{-1}\beta \left( r\right) }{\left( \log ^{\left[ q-1\right]
}r\right) ^{\rho _{\alpha }^{\left( p,q\right) }\left( \beta \right) }}%
=\sigma _{\alpha }^{\left( p,q\right) }\left( \beta \right) ~.
\end{equation*}%
Thus the theorem follows.
\end{proof}

\begin{remark}
If we consider $\alpha (x)=M_{g}(x)$ and $\beta (x)=M_{f}(x)$ where $f$ and $%
g$ are any two entire functions with index-pairs $\left( m,q\right) $ and $%
\left( m,p\right) $ respectively where $p,q,m$ are positive integers such
that $m\geq \max (p,q),$ then the above results reduces for the relative $%
\left( p,q\right) $-th growth indicators such as relative $\left( p,q\right)
$-th type, relative $\left( p,q\right) $-th weak type etc. of an entire
function $f$ with respect to another entire function $g$.
\end{remark}

\begin{remark}
If we take $\alpha (x)=T_{g}(x)$ and $\beta (x)=T_{f}(x)$ where $f$ be a
meromorphic function and $g$ be any entire function with index-pairs $\left(
m,q\right) $ and $\left( m,p\right) $ respectively where $p,q,m$ are
positive integers such that $m\geq \max (p,q),$ then the above theorems
reduces for relative $\left( p,q\right) $-th growth indicators such as
relative $\left( p,q\right) $-th type, relative $\left( p,q\right) $-th weak
type etc. of a meromorphic function $f$ with respect to an entire function $g
$.
\end{remark}


\begin{thebibliography}{9}
\bibitem{1} L. Bernal : Crecimiento relativo de funciones enteras. Contribuci%
\'{o}n al estudio de lasfunciones enteras con \'{\i}ndice exponencial
finito, Doctoral Dissertation, University of Seville, Spain, 1984.

\bibitem{2} L. Bernal : Orden relativo de crecimiento de funciones enteras,
Collect. Math. Vol. \textbf{39 }(1988), pp. 209-229.

\bibitem{3} L. Debnath, S. K. Datta, T. Biswas and A. Kar: Growth of
meromorphic functions depending on (p,q)-th relative order, Facta
Universititatis series: Mathematics and Informatics, Vol. 31, No 3 (2016),
pp. 691-705.

\bibitem{4} W.K. Hayman : Meromorphic Functions, The Clarendon Press, Oxford
(1964).

\bibitem{5} L. M. S. Ruiz, S. K. Datta, T. Biswas and G. K. Mondal: On the
(p,q)-th relative order oriented growth properties of entire functions,
Abstract and Applied Analysis, Vol.2014, Article ID 826137, 8 pages,
http://dx.doi.org/10.1155/2014/826137.

\bibitem{6} G. Valiron : Lectures on the General Theory of Integral
Functions, Chelsea Publishing Company,1949.
\end{thebibliography}
\end{document}